\documentclass[12pt,letterpaper]{amsart}
\usepackage{geometry,mathpazo}
\usepackage{amsmath,amssymb,amsfonts,amsthm}
\usepackage{url,setspace}
\newtheorem{theorem}{Theorem}
\newtheorem*{theorem*}{Theorem}
\newtheorem{proposition}{Proposition}
\newtheorem{lemma}{Lemma}
\newtheorem{corollary}{Corollary}
\newtheorem{question}{Question}
\theoremstyle{remark}
\newtheorem{remark}{Remark}
\newcommand{\C}{\mathbb{C}}
\newcommand{\zb}{\overline{z}}
\newcommand{\D}{\Omega}
\newcommand{\Dc}{\overline{\Omega}}
\newcommand{\dbar}{\overline{\partial}}
\usepackage{hyperref}
\onehalfspace

\title[Compactness of the $\dbar$-Neumann operator and the commutators]
{Compactness of the $\dbar$-Neumann operator and commutators of the Bergman
projection with continuous functions}

\author{Mehmet \c{C}el\.ik}
\address[Mehmet \c{C}elik]{University of North Texas at Dallas, Department of
Mathematics \& Information Sciences,  7400 University Hills Blvd., 
Dallas, TX 75241}
\email{mehmet.celik@unt.edu}

\author{S\"{o}nmez \c{S}ahuto\u{g}lu}
\address[S\"{o}nmez \c{S}ahuto\u{g}lu]{ University of Toledo, Department of
Mathematics \& Statistics, 2801 W. Bancroft, Toledo, OH 43606, USA}
\email{sonmez.sahutoglu@utoledo.edu}

\date{\today}

\begin{document}

\begin{abstract}
Let $\D$ be a bounded pseudoconvex domain in $\C^n, n\geq 2, 0\leq p\leq n,$ and
$1\leq q\leq n-1.$ We show that compactness of the $\dbar$-Neumann operator, $N_{p,q+1},$
on square integrable $(p,q+1)$-forms is equivalent to compactness of the commutators
$[P_{p,q}, \zb_j]$ on square integrable $\dbar$-closed $(p,q)$-forms for $1\leq j\leq n$
where $P_{p,q}$ is the Bergman projection on $(p,q)$-forms. We also show that  compactness
of the commutator of the Bergman projection with bounded functions percolates up
in the $\dbar$-complex on $\dbar$-closed forms and square integrable holomorphic
forms.
\end{abstract}

\maketitle

The purpose of this paper is to characterize compactness of the $\dbar$-Neumann
operator on square integrable  $(p,q)$-forms. Theorem \ref{Thm1} provides six
equivalent statements, for $q$ at least $2$, on  bounded pseudoconvex domains. 
However, the important special case of functions, namely $(0,0)$-forms, remains
open. In Remark  \ref{Remq=0}  we discuss why our proof breaks down in this
case.

Compactness results in the $\dbar$-Neumann problem have a long history;
we refer to a recent book by Straube \cite{StraubeBook} for a detailed 
discussion. We note here that most known results provide conditions for
compactness in terms of the boundary geometry. It is also useful to characterize
compactness in functional analytic terms. For example, Catlin and D'Angelo
\cite[Theorem 1]{CatlinD'Angelo97} used compactness of the commutators
$[P,\phi]$ of the Bergman projection $P$ and certain multiplication operators
$\phi$ in conjunction with a complex variables analogue of Hilbert's 17th
problem (see also \cite{D'Angelo11}). In the same paper they showed that
compactness of $N_{0,1}$ implies that the commutators $[P,M]$ are compact for
all tangential pseudodifferential operators $M$ of order 0. D'Angelo then asked
the following question: 

\begin{question}\label{Question}
Can one characterize compactness of the $\dbar$-Neumann operator in 
terms of commutators $[P,\phi]$?
\end{question}

This question is appealing because of its connection to operator theory as well. 
Let $A^2(\Omega)$ be the space of square integrable holomorphic functions on 
$\Omega$. The Hankel operator $H_{\phi}: A^2(\Omega) \to L^2(\Omega)$, with 
a bounded symbol $\phi$, is defined by $H_{\phi}(f) = (I-P)(\phi f)$. Using Kohn's 
formula, $P=I-\dbar^*N_{0,1}\dbar,$ one  obtains that 
$H_{\phi}f=\dbar^*N_{0,1}\dbar(f\phi).$ Using this formula,  \v{C}u\v{c}kovi\'c
and \c{S}ahuto\u{g}lu \cite{CuckovicSahutoglu09} studied how boundary geometry
interacts with Hankel operators. They showed that, on smooth bounded convex
domains in $\C^2,$ compactness of  $H_{\phi}$ can be characterized by the
behavior of $\phi$ on analytic discs in the boundary. Here $\phi$ is smooth up
to the boundary. It would be interesting to know if this characterization
still holds in higher dimensions. 

On convex domains the relation between the compactness of the commutators and
of the $\dbar$-Neumann operator has been fairly well understood: if $\D$ is a 
bounded convex domain, then compactness of  $N_{p,q+1}$ is equivalent to
compactness of  the commutators $[P_{p,q},\phi]$ on the space of $(p,q)$-forms
with square integrable holomorphic coefficients, for all functions $\phi$
continuous on $\Dc$ (see \cite[Remark (ii) in Section 4.1]{StraubeBook}).

Theorem \ref{Thm1} fails on non-pseudoconvex domains. In our previous paper 
\cite{CelikSahutoglu2012} we  constructed a smooth bounded non-pseudoconvex 
domain in $\C^n$ (for a given $n\geq 3$) for which the commutators $[P,\phi]$ are 
compact (on square integrable functions)  for all $\phi$ continuous on the closure 
of the domain, yet the $\dbar$-Neumann operator $N_{0,1}$ is not compact. In this 
paper we will consider the important pseudoconvex case and establish a decisive 
result on forms in Theorem \ref{Thm1}.

We would like to thank  \v{Z}eljko \v{C}u\v{c}kovi\'c, John D'Angelo, and Emil
Straube for valuable comments on a preliminary version of this manuscript. We also 
thank the referee for comments that improved the exposition of the paper and for 
suggesting Lemma \ref{LemAlgebra}.  

\section{Background and Main Results}

Let $\D$ be a bounded pseudoconvex domain in $\C^n.$ We denote the 
square-integrable $(p,q)$-forms on $\D$ by $L^2_{(p,q)}(\D)$ and the subspace of 
$\dbar$-closed forms by $K^2_{(p,q)}(\D)$. The $\dbar$-Neumann
operator, $N_{p,q},$  is defined as the solution operator for $\Box_{p,q} u=v$
where $\Box_{p,q}=\dbar\dbar^*+\dbar^*\dbar$ on $L^2_{(p,q)}(\D)$ and $\dbar^*$
is the Hilbert space adjoint of $\dbar$. H\"{o}rmander \cite{Hormander65} showed
that $N_{p,q}$ is a bounded operator when $\D$ is a bounded pseudoconvex
domain. Kohn in \cite{Kohn63} connected the Bergman projection $P_{p,q}$ on
$(p,q)$-forms with the $\dbar$-Neumann operator by the formula
\[P_{p,q}=I-\dbar^* N_{p,q+1}\dbar.\]
We note that $P_{p,q}$ is the orthogonal projection onto $K^2_{(p,q)}(\D),$ the
operator  $P_{0,0}$ is the classical Bergman projection $P$, and
$K^2_{(0,0)}(\D)$ (also denoted as $A^2(\D)$) is called the Bergman space. We
refer the reader to \cite{ChenShawBook,StraubeBook} for more information about
the $\dbar$-Neumann problem and related issues.

The commutators of the Bergman projection with multiplication operators can
also be written in terms of the $\dbar$-Neumann operator. Let $f\in K^2_{(p,q)}(\D)$ 
and $\phi\in C^{1}(\Dc)$. Then the equality 
\[[P_{p,q},\phi]f= -\dbar^*N_{p,q+1}\dbar\phi\wedge f\]
follows easily from Kohn's formula. Furthermore, compactness
of  $N_{p,q+1}$ on $L^2_{(p,q+1)}(\D)$ implies that $[P_{p,q},\phi]$ is
compact on $L^2_{(p,q)}(\D)$ for all  $\phi\in C(\Dc)$  (see
\cite[Propositions 4.1 and 4.2]{StraubeBook}).  In the following theorem we
show that the the converse is true when $q\geq 1.$

\begin{theorem} \label{Thm1}
Let $\D$ be a bounded pseudoconvex domain in $\C^n, n\geq 2, 0\leq p\leq n,$ and
$1\leq q\leq n-1.$  Then the following are equivalent:
\begin{itemize}
\item[i.] $N_{p,q+1}$ is compact on $L^2_{(p,q+1)}(\D),$
\item[ii.] $\dbar^*N_{p,q+1}$ is compact on $K^2_{(p,q+1)}(\D),$ 
\item[iii.] $[P_{p,q},\zb_j]$ is compact on $K^2_{(p,q)}(\D)$ for all 
$1\leq j\leq n,$
\item[iv.] $[P_{p,q},\zb_j]$ is compact on $L^2_{(p,q)}(\D)$ for all 
$1\leq j\leq n,$
\item[v.] $[P_{p,q},\phi ]$ is compact on $L^2_{(p,q)}(\D)$ for all 
$\phi\in C(\Dc),$
\item[vi.] $[P_{p,q},\phi ]$ is compact on $K^2_{(p,q)}(\D)$ for all 
$\phi\in C(\Dc).$
\end{itemize}
\end{theorem} 

The most important implications in the Theorem \ref{Thm1} are  iii. implies i. and iv. 
implies v.. The rest of the implications are either known or easy: 
i. implies ii. is known \cite[Proposition 4.2]{StraubeBook}; ii. implies iii. is easy; 
ii. implies iv. follows from \cite[Proposition 4.1]{StraubeBook}; v. implies vi. and 
vi. implies iii. are obvious. 

We would like to mention that Haslinger in \cite[Theorem 3]{Haslinger08}
proved equivalence of iii., iv., v., and vi. in Theorem \ref{Thm1} only in case of
$p=q=0.$ 

One consequence of Theorem \ref{Thm1} is the following observation: To
conclude that $N_{p,q}$ is compact it is enough to verify compactness of
$\dbar^*N_{p,q}$ only, instead of verifying compactness of both $\dbar^*N_{p,q}$
and $\dbar^*N_{p,q+1}$  (see Proposition \ref{Prop1} in the next section).

\begin{remark}\label{CompactnesN1vsN0}
We note that compactness of $N_{p,0}$ on the orthogonal complement of
$A^{2}(\D)$ is equivalent to compactness of $N_{p,1}$ on $L^2_{(p,1)}(\D).$ This
can be seen as follows:  
$N_{p,0}=\left(\dbar^* N_{p,1}\right)\left(\dbar^*N_{p,1}\right)^*.$ This formula shows
that compactness of $N_{p,0}$ implies compactness of $\dbar^* N_{p,1}.$ Then 
Lemma \ref{LemCompPerculate} implies that $\dbar^* N_{p,2}$ compact. Finally, Range's
Formula  (see \cite[p. 77]{StraubeBook} and \cite{Range84}) for $q=1$
\[N_{p,1}=\left(\dbar^* N_{p,1}\right)^*\left(\dbar^* N_{p,1}\right)+
\left(\dbar^* N_{p,2}\right)\left(\dbar^* N_{p,2}\right)^*\]
implies that $N_{p,1}$ is compact. On the other hand, compactness  $N_{p,1}$ implies 
compactness of $\dbar^* N_{p,1}.$ In turn, the formula  $N_{p,0}=\left(\dbar^*
N_{p,1}\right)\left(\dbar^*N_{p,1}\right)^*$ shows that, in this case, $N_{p,0}$ is
compact on the orthogonal complement of $A^{2}(\D)$.
\end{remark}
 
Compactness of the $\dbar$-Neumann operator percolates up in the $\dbar$-complex
(see Remark \ref{RemPercolate}).  The following theorem shows that the same is
true for the commutator of the Bergman projection with a function continuous on the
closure of the domain. 

\begin{theorem}\label{Thm2}
Let $\D$ be a bounded pseudoconvex domain in $\C^n$ for $n\geq 2, 0\leq p\leq
n,1\leq q\leq n-1,$ and $\phi\in L^{\infty}(\D).$ Assume that $[P_{p,q},\phi]$
is compact on $K^2_{(p,q)}(\D).$ Then $[P_{p,q+1},\phi]$ is compact on
$K^2_{(p,q+1)}(\D)$.
\end{theorem}

\begin{remark}
Compactness of $[P_{p,q},\phi]$ on $K^2_{(p,q)}(\D),$ for a fixed $\phi$, does
not necessarily imply compactness of $[P_{p,q},\phi]$ on $L^2_{(p,q)}(\D),$ in
general. One can construct a counterexample as follows: Let $\D\subset \C^n$ be
the polydisk, $0\leq p\leq n,0\leq q\leq n-2,$ and consider $\phi=z_j$ for
$1\leq j\leq n.$ Then $[P_{p,q},z_j]=0$ on $K^2_{(p,q)}(\D),$ hence it is compact. On
the other hand, if $[P_{p,q},z_j]$ were compact  on $L^2_{(p,q)}(\D )$, then the
adjoint $[P_{p,q},z_j]^*=-[P_{p,q},\zb_j]$ would be compact on $L^2_{(p,q)}(\D)$
as well. We note that 
\[-[P_{p,q},\zb_j] g= \dbar^*N_{pq+1}(d\zb_j\wedge g)\]
for $g\in K^2_{(p,q)}(\D).$ In addition, one can show that any 
$f\in A^2_{(p,q+1)}(\D)$ can be written as
$f=\sum_{j=1}^nf_j\wedge d\zb_j$ where $f_j\in A^2_{(p,q)}(\D)$ and
$\sum_{j=1}^n\|f_j\|^2=\|f\|^2$ (see the proof of Corollary \ref{Cor2} in the
next section). Then we have
\[\dbar^*N_{p,q+1}f=(-1)^{p+q+1}\sum_{j=1}^n[P_{p,q},\zb_j]f_j.\]
Therefore, if $[P_{p,q},z_j]$ were compact on $L^2_{(p,q)}(\D)$ for 
$1\leq j\leq n$ the operator $\dbar^*N_{p,q+1}$ would be compact on
$A^2_{(p,q+1)}(\D).$ However, this contradicts  \cite[Theorem 1.1]{FuStraube98}.
We note that even though \cite[Theorem 1.1]{FuStraube98} is stated on 
$L^2_{(0,q)}(\D)$ the proof only uses forms with holomorphic coefficients and
the proof is valid on $(p,q)$-forms as well 
(see remarks 2 and 4 in \cite[pg. 638]{FuStraube98}).
\end{remark}

\section{Proofs of Theorem \ref{Thm1} and \ref{Thm2}} \label{SecProofs} 
The proof of Theorem \ref{Thm1} will be based on several lemmas whose proofs are 
standard  if one is familiar with the basics of the $\dbar$-Neumann problem. 

\begin{lemma}\label{LemDbarClosed}
Let $\D$ be a bounded pseudoconvex domain in $\C^n$ for $n\geq 2$ and 
$g \in K^2_{(p,q+1)}(\D)$ where $0\leq p\leq n$ and $1\leq q\leq n-1.$  Then 
there exist $g_j\in K^2_{(p,q)}(\D)$ for $1\leq j\leq n$ such that 
\[g=\sum_{j=1}^{n} g_j\wedge d\zb_j \text{ and } \sum_{j=1}^n\|g_j\|\lesssim \|g\|.\]
\end{lemma}
\begin{proof}
Let $0\leq p\leq n,1\leq q\leq n-1,$ and
$f=\sideset{}{'}\sum_{|I|=p,|J|=q}f_{IJ}dz_I\wedge d\zb_J=\dbar^*N_{p,q+1}g.$ Then we can write 
\[f= \sum_{j=1}^nf_j\wedge d\zb_j\]
where $f_j$'s are square integrable  $(p,q-1)$-forms so that  there are no common terms
between  $f_j\wedge d\zb_j$ and $f_k\wedge d\zb_k$ if $j\neq k.$ This can be done as
follows:  Let $\vee$ denote the adjoint of the exterior multiplication. That is, if $f$
is a $(p,q)$-form $d\zb_j \vee f$ is a $(p,q-1)$-form such that 
$\langle h\wedge d\zb_j,f \rangle =\langle h,d\zb_j\vee f \rangle$ for all 
$h\in C^{\infty}_{(p,q-1)}(\C^n).$ Then we define 
\begin{align*}
f_1&=d\zb_1\vee f\\
f_j&=d\zb_j\vee \left(f-\sum_{k=1}^{j-1}f_k\wedge d\zb_k\right)\text{ for } j=2,3, \ldots, n.
\end{align*}
Namely, $f_1$ is defined by collecting all terms that contain $d\zb_1$ and writing
their sum as $f_1\wedge d\zb_1.$ Then we define $f_2$ by collecting the terms in
$f-f_1\wedge d\zb_1$ with $d\zb_2$ and writing their sum as $f_2\wedge d\zb_2$ etc. Since
$\dbar g=0$ and $f$ is in the range of $\dbar^*,$ we have $\dbar f=g$ and $\dbar^* f=0.$
So $f$ is in the domains of $\dbar$ and $\dbar^*.$ Also since $f_j$ consists of terms
$f_{IJ}$ for some $|I|=p$ and $|J|=q$,  ``bar" derivatives of $f_j$'s come from ``bar"
derivatives of $f.$ Then  
\[\sum_{j,k=1}^n \left\|\frac{\partial f_j}{\partial \zb_k}\right\| \lesssim
\sum_{|I|=p,|J|=q}\sum_{k=1}^n  \left\|\frac{\partial f_{IJ}}{\partial \zb_k}\right\|. \] 
This fact together with \cite[Corollary 2.13]{StraubeBook} imply that
\[\sum_{j,k=1}^n \left\|\frac{\partial f_j}{\partial \zb_k}\right\| \lesssim
\|\dbar f\|+\|\dbar^*f\|=\|g\|.\]
Hence, $\|\dbar f_j\|\lesssim \|g\|$ for every $j$ and 
\[g=\dbar f=\sum_{j=1}^n\dbar f_j\wedge d\zb_j.\]
Therefore, if we define $g_j=\dbar f_j$ we have  
$g=\sum_{j=1}^{n} g_j\wedge d\zb_j$ and
$\sum_{j=1}^n\|g_j\|\lesssim \|g\|.$
\end{proof}

\begin{remark} \label{Remq=0}
Our proof of Theorem \ref{Thm1} depends on Lemma \ref{LemDbarClosed}. 
Lemma \ref{LemDbarClosed} fails for $q=0$ and hence the implication iii.  
implies i.  is not known when $q=0$.  One can see
that Lemma \ref{LemDbarClosed} is not true for $q=0$ as follows: Let
$g=\phi(z_1)d\zb_1$  where  $\phi$ is a non-holomorphic function that is
smooth on $\Dc.$ Then $g$ is  $\dbar$-closed but there is no holomorphic
function $g_1$ such that $g=g_1d\zb_1.$ Also it is interesting that the proof of
Lemma \ref{LemDbarClosed} requires the existence of $(p,q-1)$-forms (that is,
two form-levels below the starting form-level). 
\end{remark}

The following lemma shows that the converse of \cite[Proposition 4.1]{StraubeBook} is 
true.

\begin{lemma}\label{LemComCanonical} 
Let $\D$ be a bounded pseudoconvex domain in $\C^n, n\geq 2,$ and 
$0\leq p\leq n,1\leq q\leq n-1.$ Then $[P_{p,q},\zb_j]$ is compact on $K^2_{(p,q)}(\D)$
for all  $1\leq j\leq n$ if and only if $\dbar^*N_{p,q+1}$ is compact on
$L^2_{(p,q+1)}(\D)$.
\end{lemma}

\begin{proof} 
We only need to prove one direction as compactness of $\dbar^*N_{p,q+1}$ on
$L^2_{(p,q+1)}(\D)$ implies that $[P_{p,q},\zb_j]$ is compact on $\dbar$-closed
$(p,q)$-forms for $1\leq j\leq n$ (see \cite[Proposition 4.1]{StraubeBook}). 

To prove the other direction, assume that $\{g^k\}$ is a bounded sequence in
$K^2_{p,q+1}(\D)$. Then Lemma \ref{LemDbarClosed} implies that for each $k$ there exist 
$\dbar$-closed $(p,q)$-forms $g^k_j$ for $1\leq j\leq n$ such that $g^k=\sum_{j=1}^{n}
g^k_{j}\wedge d\zb_{j}$ and $\sum_{j=1}^{n} \|g^k_{j}\|\lesssim \|g^k\|.$ 
Then we have 
\[\dbar^*N_{p,q+1}(g^k)=(-1)^{p+q+1}\sum_{j=1}^n[P_{p,q},\zb_j](g^k_j).\]
Furthermore, if  $[P_{p,q},\zb_j]$ is compact on $\dbar$-closed $(p,q)$-forms for 
$1\leq j\leq n,$  sequences $\{[P_{p,q},\zb_j](g^k_j)\}$ have convergent subsequences for
each $j.$ Hence $\dbar^*N_{p,q+1}$ is compact on $K_{(p,q+1)}^{2}(\D)$. 
On the other hand, compactness of $\dbar^*N_{p,q+1}$ on $\dbar$-closed forms is equivalent
to compactness of $\dbar^*N_{p,q+1}$ on $L^2_{(p,q+1)}(\D)$ as $\dbar^*N_{p,q+1}$ vanishes
on the orthogonal complement of $K_{(p,q+1)}^{2}(\D).$
\end{proof}

\begin{remark}\label{InterestingObservation1}
It is interesting to observe that if $q=n-1$ then the commutators $[P_{p,n-1},\zb_j]$ are
compact on $K^2_{(p,n-1)}(\D)$  on  any bounded  pseudoconvex domain  $\D$ with
sufficiently smooth boundary. This is  a consequence of the fact that the
$\dbar$-Neumann problem on $(p,n)$-forms is the classical Dirichlet problem. See
\cite[Remark after Corollary 5.1.7.]{ChenShawBook})
\end{remark}

\begin{remark}
It is  unknown whether Lemma \ref{LemComCanonical} is true when the  commutators
are restricted to $A_{(p,q)}^2(\D)$ for $0\leq p\leq n,$ and $0\leq q\leq n-1$. That is, it is 
not known whether compactness of $[P_{p,q},\zb_j]$ on  $A_{(p,q)}^2(\D)$ for all $j$
(equivalent to compactness of $\dbar^*N_{p,q+1}$ on  $A_{(p,q+1)}^2(\D)$ by 
\cite[Remark (ii) in Section 4.1]{StraubeBook}) imply compactness of $\dbar^*N_{p,q+1}$
on $K^2_{(p,q+1)}(\D)$. 
\end{remark}

\begin{lemma}\label{LemCompPerculate} 
Let $\D$ be a bounded pseudoconvex domain in $\C^n, n\geq 2, 0\leq p\leq n,$ and
$1\leq q\leq n-1.$ Then compactness of $\dbar^*N_{p,q}$ on $K^2_{(p,q)}(\D)$ implies
that $\dbar^*N_{p,q+1}$ is compact on $K^2_{(p,q+1)}(\D).$ 
\end{lemma}

\begin{proof}
Let $\{g^k\}$ be a bounded sequence of  $\dbar$-closed $(p,q+1)$-forms. Then by 
Lemma \ref{LemDbarClosed} there exist  $\dbar$-closed $(p,q)$-forms  $g^{k}_{j}$'s 
such that  
$g^k=\sum_{j=1}^ng^k_j\wedge d\zb_j$ and $\sum_{j=1}^n \|g^k_j\|\lesssim \|g^k\|.$  

Let us define $f^k=\sum_{j=1}^n\dbar^*N_{p,q}(g^k_j)\wedge d\zb_j.$  Then $\dbar f^k=g^k$
and  compactness  of $\dbar^*N_{p,q}$ implies that $\{f^k\}$ has a  convergent
subsequence.  Therefore, $\dbar$ has a compact solution operator on $(p,q+1)$-forms.
Hence, the canonical solution operator,  $\dbar^*N_{p,q+1},$ is compact on
$K^2_{(p,q+1)}(\D).$ 
\end{proof}

\begin{remark}\label{RemPercolate}
One corollary of Lemma \ref{LemCompPerculate} is the well known fact that
compactness of $N_{p,q}$  implies compactness of $N_{p,q+1}$ (see 
\cite[Proposition 4.5]{StraubeBook}). One can see this using Lemma
\ref{LemCompPerculate} together  with  Range's formula
(\cite{Range84}): 
\[N_{p,q}=(\dbar^*N_{p,q})^*\dbar^*N_{p,q}+\dbar^*N_{p,q+1}(\dbar^*N_{p,q+1})^*.\] 
\end{remark}

\begin{lemma}\label{LemAlgebra}
Let $A,B,C$ be bounded operators on a Hilbert space. Then $[A,BC] = [A,B] C + B[A,C]$. 
Furthermore, if $[A,B]$ and $[A,C]$ are compact, then so is $[A,BC]$.
\end{lemma}
\begin{proof}
The equality $[A,BC] = [A,B] C + B[A,C]$ is immediate as 
\[[A,B] C + B[A,C] = ABC-BAC + BAC - BCA = [A,BC]. \]
Hence if $[A,B]$ and $[A,C]$ are compact,  then so are $[A,B] C$ and 
$B[A,C]$, and the result follows. 
\end{proof}

Let us define 
\[\Gamma(p,q)=\{\phi\in C(\Dc):[P_{p,q},\phi]: L^2_{(p,q)}(\D)\to L^2_{(p,q)}(\D)
\text{ is compact}\}.\]

\begin{corollary}\label{CorAlgebra}
Let $\D$ be a bounded domain in $\C^n.$ Then $\Gamma(p,q)$  is a $C^*$ subalgebra
(with identity) of $C(\Dc).$ 
\end{corollary}
\begin{proof} 
The fact that compactness is preserved under the operator norm topology implies that 
$\Gamma(p,q)$ is a closed subspace of  $C(\Dc).$   Lemma \ref{LemAlgebra} implies that
$\Gamma(p,q)$  is a subalgebra  of $C(\Dc)$ and $1\in \Gamma(p,q)$ because
$[P_{p,q},1]=0.$ Finally, the fact that $\Gamma (p,q)$ is closed under conjugation
follows from the formula  $[P_{p,q},\phi]^*=-[P_{p,q},\overline{\phi}]$ and the
fact that an operator is compact if and only if so is its adjoint.
\end{proof}

We note that Corollary \ref{CorAlgebra} gives a characterization for compactness of
$N_{p,q}$ for $q\geq 1$ in terms of $C^*$ algebras. Namely, for $1\leq q\leq n-1$ the
operator  $N_{p,q+1}$ is compact if and only if $\Gamma (p,q)=C(\Dc)$ (compare to
\cite{Salinas91}).

We will use the followig well known fact (see, for example, \cite[Proposition 4.2]{StraubeBook}) 
in the proof of Theorem \ref{Thm1}
\begin{proposition}\label{Prop1}
Let $\D$ be a bounded pseudoconvex domain in $\C^n$ and $0\leq p\leq n, 1\leq q\leq n$. 
Then $N_{p,q}$ is compact on $L^2_{(p,q)}(\D)$ if and only if  $\dbar^*N_{p,q}$ and 
$\dbar^*N_{p,q+1}$ are compact on $K^2_{(p,q)}(\D)$ and $K^2_{(p,q+1)}(\D)$, respectively. 
\end{proposition}

\begin{proof}[Proof of Theorem \ref{Thm1}]
We note that i. $\Rightarrow$ ii. is known (see \cite[Proposition 4.2]{StraubeBook}) and
ii. $\Rightarrow$ iii. is easy because   $[P_{p,q},\zb_j]
(f)=-\dbar^*N_{p,q+1}(d\zb_j \wedge f)$ for $f\in K^2_{(p,q)}(\D).$ 

Now we will prove iii. $\Rightarrow$ i.: Compactness of $[P_{p,q},\zb_j]$ on
$K^2_{(p,q)}(\D)$ for $1\leq j\leq n,$ by Lemma \ref{LemComCanonical}, is equivalent
to compactness of $\dbar^*N_{p,q+1}$ on $K^2_{(p,q+1)}(\D)$. On the other hand,
if $q=n-1$ then compactness of  $\dbar^*N_{p,q+1}$ is equivalent to compactness
of $N_{p,q+1}.$ In case $1\leq q\leq n-2,$ Lemma \ref{LemCompPerculate} implies
that compactness of $\dbar^*N_{p,q+1}$ on $K^2_{(p,q+1)}(\D)$ implies 
compactness of $\dbar^*N_{p,q+2}$ on $K^2_{(p,q+2)}(\D)$. Then 
Proposition \ref{Prop1} implies that $N_{p,q+1}$ is compact on
$L^2_{(p,q+1)}(\D).$

The implication ii. $\Rightarrow$ iv. follows from the fact that 
$\dbar^*N_{p,q+1}$ vanishes on the orthogonal complement of
$K^2_{(p,q+1)}(\Omega)$ and \cite[Proposition 4.1]{StraubeBook}.

We prove the implication iv. $\Rightarrow$ v. as follows:  By Corollary \ref{CorAlgebra}
we know that $\Gamma(p,q)$ is a closed subalgebra with identity (of $C(\Dc)$) that is
closed under conjugation. Then the assumption that the commutators $[P_{p,q},\zb_j]$ are
compact on $L^2_{(p,q)}(\D),$ for $1\leq j\leq n,$ together with the  Stone-Weierstrass 
Theorem imply that $\Gamma (p,q)=C(\Dc).$ That is $[P_{p,q},\phi]$ is compact for all
$\phi\in C(\Dc).$ 

Finally, the implications v. $\Rightarrow$ vi.  and   vi. $\Rightarrow$ iii. are obvious.
\end{proof}

\begin{proof}[Proof of Theorem\ref{Thm2}] 
If $q=n-1$ then  $K^2_{(p,q+1)}(\D)= L^2_{(p,n)}(\D)$ and $[P_{p,q+1},\phi]$ is the zero
operator, hence compact. So for the rest of the proof we may assume that $n\geq 3$ and
$1\leq q\leq n-2.$ 

Let $g \in K^2_{(p,q+1)}(\D)$. Then Lemma \ref{LemDbarClosed} implies that there exist
$g_j\in K^2_{(p,q)}(\D)$ for $1\leq j\leq n$ such that 
\[g=\sum_{j=1}^{n} g_j\wedge d\zb_j \text{ and } \sum_{j=1}^n\|g_j\|\lesssim
\|g\|.\]
Now  we will show that 
\begin{align}\label{EqnDbarClosed}
[P_{p,q+1},\phi]g=(I-P_{p,q+1})\left(\sum_{j=1}^n ([P_{p,q},\phi]g_j)\wedge
d\zb_j\right).
\end{align}
Since both sides of \eqref{EqnDbarClosed} are orthogonal to $K^2_{(p,q+1)}(\D)$ we
only need to show that for any $h\in L^2_{(p,q+1)}(\D)$ that is orthogonal to
$K^2_{(p,q+1)}(\D)$ we have 
\[ \left\langle [P_{p,q+1},\phi]g-(I-P_{p,q+1})\left(\sum_{j=1}^n
([P_{p,q},\phi]g_j)\wedge d\zb_j\right),h\right\rangle =0\] 
where $\langle,.,\rangle$ denotes the inner product on $L^2_{(p,q+1)}(\D).$ One can
compute that 
\begin{align*}
 \left\langle [P_{p,q+1},\phi]g-(I-P_{p,q+1})\left(\sum_{j=1}^n
([P_{p,q},\phi]g_j)\wedge d\zb_j\right),h \right\rangle = &-\langle \phi g, h\rangle -
\left\langle \sum_{j=1}^n P_{p,q}(\phi g_j)\wedge d\zb_j , h\right\rangle\\
&+\left\langle \sum_{j=1}^n\phi g_j\wedge d\zb_j,h\right\rangle \\
=&-\left\langle \sum_{j=1}^n P_{p,q}(\phi g_j)\wedge d\zb_j , h\right\rangle.
\end{align*}
The fact that $\dbar (f\wedge d\zb_j)=(\dbar f)\wedge d\zb_j$  implies that the
$(p,q+1)$-forms $P_{p,q}(\phi g_j)\wedge d\zb_j$ are $\dbar$-closed for  $j=1,\ldots,n.$
Therefore, we have $\left\langle \sum_{j=1}^n P_{p,q}(\phi g_j)\wedge d\zb_j ,
h\right\rangle=0$ and the equality \eqref{EqnDbarClosed} is proven. 

Let $\{g^k\}\subset K^2_{(p,q+1)}(\D)$ be a bounded sequence. Then Lemma
\ref{LemDbarClosed} implies that for each $k$ and $1\leq j\leq n$ there exists 
$g_j^k\in K^2_{(p,q)}(\D)$  such that $g^k=\sum_{j=1}^ng^k_j\wedge d\zb_j$ and
$\sum_{j=1}^n\|g_j^k\|\lesssim \|g^k\|.$ Furthermore, compactness of
$[P_{p,q},\phi]$ on $K^2_{(p,q)}(\D)$ implies that for each $1\leq j\leq n$ the
sequence $\{[P_{p,q},\phi]g_j^k\}$ has a convergent subsequence.  Now using
\eqref{EqnDbarClosed} we conclude that the sequence $\{[P_{p,q+1},\phi]g^k\}$
has a convergent subsequence.  Hence $[P_{p,q+1},\phi]$ is compact on
$K^2_{(p,q+1)}(\D).$  
\end{proof}

\begin{corollary} \label{Cor2}
Let $\D$ be a domain in $\C^n,0\leq p\leq n,0\leq q\leq n-1,$ and 
$\phi\in L^{\infty}(\D).$ Assume that $[P_{p,q},\phi]$ is compact on
$A^2_{(p,q)}(\D).$ Then $[P_{p,q+1},\phi]$ is compact on $A^2_{(p,q+1)}(\D)$.  
\end{corollary}

\begin{proof}
First of all, we note that equation \eqref{EqnDbarClosed} is valid on all domains, 
possibly unbounded or non-pseudoconvex. With this in mind, one can prove
Corollary \ref{Cor2} the same way as Theorem \ref{Thm2} except, instead of Lemma
\ref{LemDbarClosed} we need the following fact: Let 
$g\in A^2_{(p,q+1)}(\D),0\leq p\leq n,$ and  $0\leq q\leq n-1.$ Then 
for $1\leq j\leq n$ there exists $g_j\in A^2_{(p,q)}(\D)$ such that
$g=\sum_{j=1}^ng_j\wedge d\zb_j$ and $\|g\|^2=\sum_{j=1}^n\|g_j\|^2.$ 
In fact, one can define $g_1=d\zb_1\vee g$ and 
$g_j=d\zb_j\vee \left(g-\sum_{k=1}^{j-1}g_k\wedge d\zb_k\right)$
 for  $j=2,3, \ldots, n.$
\end{proof}

\begin{remark}
Corollary \ref{Cor2} is stated for any domain in $\C^n$ and  $0\leq q\leq n-1$  while 
Theorem \ref{Thm2}  is stated for bounded pseudoconvex domains and  $1\leq q\leq n-1.$ 
\end{remark}

The followig fact is included in  \cite[Remark (ii) in pg.75]{StraubeBook}.

\begin{proposition}\label{Prop2}
Let $\D$ be a bounded pseudoconvex domain in $\C^n,0\leq p\leq n,$ and 
$0\leq q\leq n-1$. Then  $[P_{p,q},\phi]$ is compact on $A^2_{(p,q)}(\D)$ for all 
$\phi\in C(\Dc)$ if and only if $N_{p,q+1}$  is compact on $A^2_{(p,q+1)}(\D)$. 
\end{proposition}

We have the following percolation result as a corollary of Corollary \ref{Cor2} and 
Proposition \ref{Prop2}.  

\begin{corollary}
Let $\D$ be a bounded pseudoconvex domain in $\C^n,0\leq p\leq n,$ and $1\leq q\leq n-1$. 
Assume that $N_{p,q}$ is compact on $A^2_{(p,q)}(\D)$. Then $N_{p,q+1}$ is compact on 
$A^2_{(p,q+1)}(\D)$. 
\end{corollary}


\end{document}